\documentclass[a4paper,12pt]{article}
\usepackage[ansinew]{inputenc}
\usepackage{amsfonts}
\usepackage{latexsym}
\usepackage{amsmath}
\usepackage{amssymb}
\usepackage{eepic}
\usepackage{graphicx}
\usepackage{float}
\usepackage{pstricks}
\usepackage{epic,eepic}

 \usepackage{xcolor}
\usepackage{tikz}
\usetikzlibrary{arrows,shapes,chains}

\newcommand{\veps}{\varepsilon}
\newcommand{\R}{\mathbb{R}}

\newcommand{\C}{\mathbb{C}}
\newcommand{\N}{\mathbb{N}}
\newcommand{\Z}{\mathbb{Z}}

\newcommand{\Ce}{\mathcal{C}}

\newcommand{\co}{\operatorname{co}}

\newcommand{\E}{\mathcal{E}}

\newtheorem{defin}{Definition}[section]

\newtheorem{theorem}[defin]{Theorem}

\newtheorem{exa}{Example}

\newtheorem{lemma}[defin]{Lemma}
\newtheorem{corollary}[defin]{Corollary}
\newenvironment{proof}
{\noindent{\it Proof.}}{\hfill $\Box$\par\vspace{2.5mm}}
\newenvironment{remark}
{\par\vspace{2.5mm}\noindent{\bf Remark.}}{\par\vspace{2.5mm}}

\newtheorem{que}{Question}

\newtheorem{pro}{Problem}

\numberwithin{equation}{section}

\makeatletter
\renewcommand{\ps@myheadings}{%
\renewcommand{\@evenhead}%
{{\rm\thepage}\hfil{\sc  J.~M.~Heittokangas and Z.-T.~Wen}\hfil}%
\renewcommand{\@oddhead}%
{\hfil{{\sc Zero distribution theory for exponential polynomials}\hfil{\rm\thepage}}}%
\renewcommand{\@evenfoot}{}%
\renewcommand{\@oddfoot}{}%
}\makeatother 

\title{\bf\Large Generalization of P\'olya's zero distribution theory for exponential polynomials, plus sharp results for asymptotic growth}
\author{Janne M.~Heittokangas
and Zhi-Tao Wen\footnote{Corresponding author.}
   }
\date{}
\begin{document}
\maketitle

\begin{abstract}
An exponential polynomial of order $q$ is an entire function of the form
    $$
    f(z)=P_1(z)e^{Q_1(z)}+\cdots +P_k(z)e^{Q_k(z)},
    $$
where the coefficients $P_j(z),Q_j(z)$ are polynomials in $z$ such that
    $$
    \max\{\deg(Q_j)\}=q.
    $$
In 1977 Steinmetz proved that the zeros of $f$ lying outside of finitely many logarithmic strips
around so called critical rays have exponent of convergence $\leq q-1$. This result does not
say anything about the zero distribution of $f$ in each individual logarithmic strip. Here, it is shown
that the asymptotic growth of the non-integrated counting function of zeros of $f$ is asymptotically comparable to $r^q$ in each logarithmic strip. The result generalizes the first order results by P\'olya
and Schwengeler from the 1920's, and it shows, among
other things, that the critical rays of $f$ are precisely the Borel
directions of order $q$ of $f$. The error terms in the asymptotic equations for $T(r,f)$ and $N(r,1/f)$ originally due to Steinmetz are also improved.


\end{abstract}

\renewcommand{\thefootnote}{}
\footnotetext[1]{Heittokangas was partially supported by the Academy of Finland \#268009
and the Magnus Ehrnrooth Foundation. Wen was partially supported by the National Natural Science Foundation of China  (No.~11771090) and the STU Scientific Research Foundation for Talents.}

\section{Background on exponential sums}

An exponential sum is an entire function of the form
    \begin{equation}\label{expsum}
    f(z)=P_0(z)e^{w_0z}+\cdots+P_n(z)e^{w_mz},
    \end{equation}
where the coefficients $P_j(z)$ are polynomials such that $P_j(z)\not\equiv 0$ for $j=1,\ldots,m$. If $P_0(z)\not \equiv 0$, we set $w_0=0$.
The constants $w_j\in\C$ are pairwise distinct, and they are called the leading coefficients of $f$.  Thus $w_j\neq 0$ for $j=1,\ldots,m$.

Every exponential sum (or more generally every exponential polynomial,
see \eqref{exp-poly} below) satisfies a linear differential equation with polynomial coefficients \cite{VPT}.
In the literature, exponential sums can also be found as solutions to
differential-difference
equations \cite{Dick0,Shapiro}, and even to linear differential equations of infinite order \cite{Dick0}.

The functions $\sin z= \frac{1}{2i}\left(e^{iz}-e^{-iz}\right)$
and $\cos z= \frac{1}{2}\left(e^{iz}+e^{-iz}\right)$ are typical
examples of exponential sums, their quotient being $\tan z$.
In 1929 Ritt \cite{Ritt} showed that if the
ratio $g$ of two exponential sums with constant coefficients is
entire, then $g$ is an exponential sum also. The proof of this
result was simplified by Lax in 1949 \cite{Lax}. In 1958 Shapiro
\cite{Shapiro} conjectured that if two exponential sums with polynomial
coefficients have infinitely many zeros in common, then they
are both multiples of some third exponential sum. This claim
is known as Shapiro's conjecture, and it remains unsolved.

We note that distribution of $a$-points and zero distribution of exponential sums are essentially the same thing because
$f$ is an exponential sum if and only if $f-a$ is.
Zero distribution of functions $f$ of the form \eqref{expsum} has been investigated for roughly a century by several authors,
starting from the case where the coefficients $P_j$ are constants and the leading coefficients $w_j$ are real and commensurable. The latter means that $w_j=\alpha p_j$,
where $\alpha\in\R$ and $p_j\in\Z$. The early developments, with references,
are surveyed in \cite{Langer}.

Zero distribution in the general case
is based on the work of P\'olya \cite{Poly,Poly2} and
Schwengeler \cite{Sch}. We summarize these findings in Theorem~A below,
for which Dickson \cite{Dick} has given a generalization in the sense
that the coefficients are asymptotically polynomials. Before stating
the actual result, we need to agree on the notation.

If $P_0(z)\not\equiv 0$, set $W_f=\{\overline{w}_0=0,\overline{w}_1,\ldots,\overline{w}_m\}$, while if $P_0(z)\equiv 0$, set $W_f=\{\overline{w}_1,\ldots,\overline{w}_m\}$. Moreover, set $W_f^0=W_f\cup\{0\}$. Thus, if $P_0(z)\not\equiv 0$, then $W_f=W_f^0$. For any finite set $G\subset\C$,
let $C(\co(G))$ denote the circumference of the convex hull $\co(G)$.
Around a given ray $\arg(z)=\varphi$, we define a logarithmic strip
    \begin{equation}\label{l-strip}
	\Lambda(\varphi,c)
	=\left\{re^{i\theta}: r>1,\, |\theta-\varphi|<c\frac{\log r}{r}\right\},
	\end{equation}
where $c>0$. Finally, we let $\theta^\bot$ denote
the argument of a ray parallel to one of the outer normals of
$\co\left(W_f\right)$. Thus the ray $\arg(z)=\theta^\bot$ is orthogonal
to a line passing through one of the sides of the polygon
$\co\left(W_f\right)$.

\bigskip
\noindent
\textbf{Theorem A}
\textnormal{(\cite{Dick,Poly,Poly2,Sch})} 
\emph{Let $f$ be an exponential sum of the form \eqref{expsum}.
Then $f$ has at most finitely many zeros outside
of the finitely many domains $\Lambda(\theta^\bot,c)$. Moreover,
let $\overline{w}_j,\overline{w}_k$ be two consecutive vertex points of
$\co\left(W_f\right)$, and let $\theta^\bot_0$ denote the argument of the
ray parallel to the outer normal of $\co\left(W_f\right)$ corresponding to its side
$[\overline{w}_j,\overline{w}_k]$. Then the
number $n(r,\Lambda)$ of zeros of $f$ in $\Lambda(\theta^\bot_0,c)\cap \{|z|<r\}$ satisfies}
    \begin{equation}\label{nPS}
    n(r,\Lambda)=|w_j-w_k|\frac{r}{2\pi}+O\left(1\right).
    \end{equation}

By summing over all sides of $\co(W_f)$, we get
$\sum |w_j-w_k|=C(\co(W_f))$. Hence Theorem~A yields
the asymptotic equations
    \begin{eqnarray}
    n(r,1/f) &=& C(\co(W_f))\frac{r}{2\pi}+O\left(1\right),\label{nq1}\\
    N(r,1/f) &=& C(\co(W_f))\frac{r}{2\pi}+O\left(\log r\right)\label{Nq1},
    \end{eqnarray}
where $r\to\infty$ without an exceptional set. Finally, Wittich
\cite{W1} has proved the asymptotic equation
    \begin{equation}\label{Tq1}
    T(r,f)=C(\co(W_f^0))\frac{r}{2\pi}+o\left(r\right).
    \end{equation}

Exponential sums can be generalized to exponential polynomials, which
are higher order entire functions of the form
    \begin{equation}\label{exp-poly}
    f(z)= P_0(z)e^{Q_0(z)} + \cdots + P_n(z)e^{Q_n(z)},
    \end{equation}
where $P_j(z)$, $Q_j(z)$ are polynomials for $0 \leq j \leq n$ .
A polynomial can be considered as a special case of an exponential polynomial.
We assume that the polynomials $Q_j(z)$ are pairwise different and
normalized such that $Q_j(0)=0$. This convention forces the
functions $g_j(z)=P_j(z)e^{Q_j(z)}$ to be linearly independent.
Let $q=\max\{\deg(Q_j)\}$ denote the order of $f$. As observed in \cite{Stein1},
we may then write $f$ in the normalized form
	\begin{equation}\label{standard-f}
    f(z)=H_0(z)e^{w_0z^q}+H_1(z)e^{w_1z^q}+\cdots  +H_m(z)e^{w_mz^q},
    \end{equation}
where the functions $H_0(z),H_1(z),\ldots,H_m(z)$ are nonvanishing exponential
polynomials of order $\leq q-p$ for some $1\leq p\leq q$, and $m\leq n$. The leading coefficients $w_j (0\leq j\leq m)$
are pairwise distinct and nonzero, except possibly $w_0$. Thus $H_0(z)\equiv 0$ precisely when
$q=\deg(Q_j)$ holds for every $j$.

Steinmetz \cite{Stein1} has proved analogues of \eqref{Nq1}
and \eqref{Tq1} for exponential polynomials. We will improve the error terms in these asymptotic
equations. So far there has been no attempt on finding a higher order
analogue of \eqref{nPS}. The main focus of this paper is to find one.
Then a higher order analogue of \eqref{nq1} will be found in two different
ways as simple consequences.

The main results are stated and further motivated
in Sections 2--3. The proofs of the main results are given in Sections~4--8.


\section{Zero distribution}


So far there has been no attempts to prove a higher order analogue
of Theorem A. The major part of this paper focuses in
proving such a result. Before stating the actual result, we need
to recall a few concepts.

If $f$ is an exponential polynomial of the form \eqref{standard-f}, then its Phragm\'en-Lindel\"of indicator function
	$$
	h_f(\theta)=\limsup_{r\to\infty}r^{-q}\log |f(re^{i\theta})|
	$$
takes the form
	\begin{equation*}
	h_f(\theta)=\max_j\left\{\Re \left(w_je^{iq\theta}\right)\right\},
	\end{equation*}
see \cite[p.~993]{GOP}. The indicator $h_f(\theta)$ is $2\pi$-periodic and differentiable everywhere except for cusps $\theta^*$ at which
	\begin{equation}\label{cusp}
	h_f(\theta^*)=\Re \left(w_je^{iq\theta^*}\right)=\Re \left(w_ke^{iq\theta^*}\right)
	\end{equation}
for a pair of indices $j,k$ such that $j\neq k$.
These cusps are called the critical angles for $f$, and the corresponding
rays are called critical rays for $f$ \cite{HITW, Stein1}.

An alternative way to find the critical angles is by
means of the convex hull $\co(W_f)$ of the conjugated leading
coefficients of $f$. Namely, let $\arg(z)=\theta^\bot_{j,k}$
be the ray that is parallel to the outer normal of $\co(W_f)$
determined by two successive
vertex points $\overline{w}_j,\overline{w}_k$ of $\co(W_f)$. Then the rays
    \begin{equation}\label{orthogonal}
    \arg(z)=(\theta^\bot_{j,k}+2\pi n)/q,\quad n\in\Z,
    \end{equation}
are critical rays for $f$, see \cite[Lemma~3.2]{HITW}.
Thus, if $\co(W_f)$ has $s$ sides/vertices, then $f$ has $sq$ critical rays all together. In the particular case when $q=1$ the critical rays and the orthogonal rays are one and the same.

Let $f$ be given in the normalized form \eqref{standard-f}, where
we suppose that $\rho(H_j)\leq q-p$ for $1\leq p\leq q$. The case $p=1$ corresponds to the standard case,
while the identity $p=q$ reduces the coefficients $H_j(z)$ to be polynomials. In \cite{HITW} it is proved that most of the zeros of $f$
are inside of modified logarithmic strips
	$$
	\Lambda_p(\theta^*,c)=\left\{z=re^{i\theta} : r>1,\ |\arg(z)-\theta^*|<c\frac{\log r}{r^p} \right\},
	$$
where $c>0$ and $\arg(z)=\theta^*$ represents one of the critical rays for $f$.
For $p\geq 2$ these domains approach the corresponding critical rays asymptotically, which does not
happen in the standard case $p=1$, see \cite[Section~2]{HITW}.
More precisely, let $N_{\Lambda_p}(r)$ denote the integrated counting function of the zeros of $f$ in $|z|\leq r$ lying outside of the union of the finitely many domains $\Lambda(\theta^*,c)$. Then
	\begin{equation}\label{Stein-thm-p}
	N_{\Lambda_p}(r)=O(r^{q-p}+\log r).
	\end{equation}
The case $p=1$ was proved earlier in \cite{Stein1}, while partial results for
the general case were proved in \cite{GM,MacColl}. For
\eqref{Stein-thm-p} to hold, the constant $c>0$ needs to be large
enough (depending on $f$). In what follows, we will assume that this
is always the case without further notice whenever discussing the
domains $\Lambda_p(\theta^*,c)$.

\begin{figure}[h]\label{Zerodomain}
 \begin{center}
    \begin{tikzpicture}[scale=0.7]
    \draw[->](-2,0)--(13,0)node[left,below]{$x$};
    \draw[->](0,-4)--(0,4)node[right]{$y$};
    \draw[thick,domain=1:11,smooth] plot(\x,{10*ln(\x)/\x^1.5})node[right]{HITW};
    \draw[domain=1:11] plot(\x,{0.9*ln(\x)})node[right]{Schwengeler};
    \draw[-,dashed][domain=0:11] plot(\x,{0.27*\x}) node[right]{P\'{o}lya};
    \draw[thick,domain=1:11,smooth] plot(\x,{-10*ln(\x)/\x^1.5});
    \draw[domain=1:11] plot(\x,{-0.9*ln(\x)});
    \draw[-,dashed][domain=0:11] plot(\x,{-0.27*\x});
    \end{tikzpicture}
    \end{center}
	\begin{quote}
    \caption{Assuming that the positive real axis is a critical ray for $f$,
    the zeros of $f$ have been considered in different regions as follows:
    P\'olya in sectors $|\arg(z)|<\veps$,
    Schwengeler in logarithmic strips $|\arg(z)|\leq c \frac{\log |z|}{|z|}$, and
    Heittokangas-Ishizaki-Tohge-Wen in modified logarithmic strips $|\arg(z)|\leq c
    \frac{\log |z|}{|z|^p}$. Here $\veps>0$ and $c>0$ are real, and $p\geq 1$
    is an integer.}
	\end{quote}
    \end{figure}
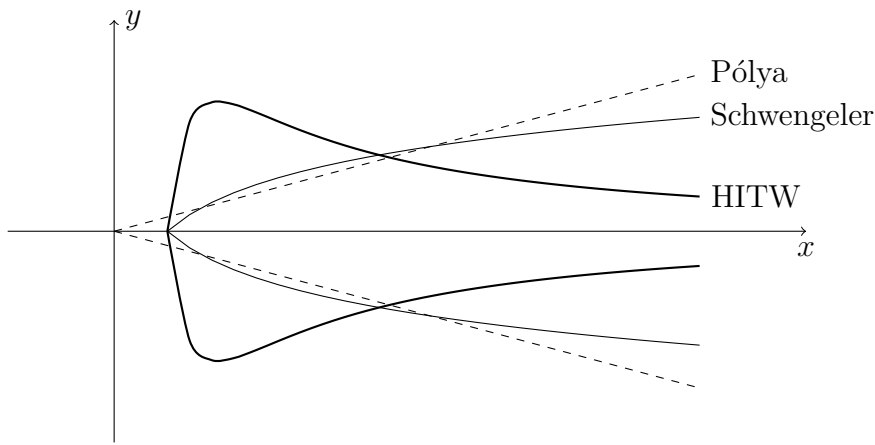

The discussion above says nothing about the zero distribution in each individual modified logarithmic strip $\Lambda_p$.
This motivates us to state and prove the following generalization of Theorem~A.

\begin{theorem}\label{length}
Let $f$ be an exponential polynomial of the form \eqref{standard-f}, where the coefficients $H_j(z)$
are exponential polynomials of growth $\rho(H_j)\leq q-p$ for some integer $1\leq p\leq q$.
Then the number $n_{\lambda_p}(r)$ of zeros of $f$ outside of the finitely many domains $\Lambda_p(\theta^*,c)$ satisfies
	$$
	n_{\lambda_p}(r)=O\left(r^{q-p}+\log r\right).
	$$
Moreover, let $\overline{w}_j,\overline{w}_k$ be two consecutive vertex points of $\co\left(W_f\right)$,
and let $\theta^\bot_0$ denote the argument of the ray parallel to the outer normal of $\co\left(W_f\right)$
corresponding to its side $[\overline{w}_j,\overline{w}_k]$.
Let $\veps>0$, and let $\theta^*$ be any critical angle for $f$ associated
with the angle $\theta^\bot_0$ by means of \eqref{orthogonal}.
Then the number of zeros $n(r,\Lambda_p)$ of $f$ in
$\Lambda_p(\theta^*,c)\cap \{|z|<r\}$ satisfies
    \begin{equation}\label{zeros-n}
    n(r,\Lambda_p)=|w_j-w_k|\frac{r^q}{2\pi}
    +O\left(r^{q-p}\log^{3+\veps}r\right).
    \end{equation}
\end{theorem}

We make some comments regarding Theorem~\ref{length}, and discuss some
immediate consequences of it.
First, according to \eqref{cusp}, at least two leading coefficients are needed to induce a critical ray.
The example
	\begin{equation}\label{example-f}
	f(z)=\sum_{j=0}^n e^{jz},\quad n\geq 2,
	\end{equation}
having critical rays at $\arg(z)=\pm \pi/2$ shows that
more than two leading coefficients can induce the same critical ray. It follows that the conjugate of any leading coefficient
inducing a critical ray must be a boundary point of
$\co\left(W_f\right)$. Indeed, if \eqref{cusp} holds, then
    $$
    \Re\left(w_je^{iq\theta^*}\right)
    =\Re\left(w_ke^{iq\theta^*}\right)
    \geq\max_{l\neq k,j}\Re\left(w_le^{iq\theta^*}\right).
    $$
From \eqref{orthogonal} we get that $\arg(z)=q\theta^*$ is an
orthogonal ray for a line through the points
$\overline{w}_j$ and $\overline{w}_k$. Hence this line must
be a support line for $\co(W_f)$, and so $\overline{w}_j$ and
$\overline{w}_k$ are boundary points of $\co(W_f)$.
However, being a boundary point of $\co\left(W_f\right)$ is not enough. In proving Theorem~\ref{length} it is
necessary that in between two consecutive $\Lambda_p$-domains precisely one
exponential term is dominant. This property is explained in more
detail in Lemma~\ref{one-leading-term} below. Only those exponential
terms that are associated with vertex points of $\co\left(W_f\right)$
have this property. For the function $f$ in \eqref{example-f}, the
dominant term in the right half-plane is $e^{nz}$, while the dominant
term in the left half-plane is $1$.

Second, keeping in mind that $\overline{w}_j,\overline{w}_k$ are two consecutive vertex points of $\co\left(W_f\right)$,  we get $\sum |w_j-w_k|=C(\co(W_f))$, where the summation is taken over all sides of $\co(W_f)$.
From \eqref{orthogonal} we see that the orthogonal ray
associated with the pair $\overline{w}_j,\overline{w}_k$ induces $q$
critical rays. This gives raise to the asymptotic equation
	\begin{equation}\label{n-asy}
    n(r,1/f) = qC(\co(W_f))\frac{r^q}{2\pi}
    +O\left(r^{q-p}\log^{3+\veps}r\right).
    \end{equation}

Third, since a domain $\Lambda_p(\theta^*,c)$ for any $p$ is essentially
contained in an arbitrary $\veps$-angle $\{z : |\arg(z)-\theta^*|<\veps\}$, it follows that the critical rays of $f$ are precisely the Borel directions of order $q$ of $f$. Note that $f$ has no Borel directions
of order $\rho\in (q-p,q)$. In the case $q=1$ the critical rays
of $f$ (which are the same as the orthogonal rays of $f$) are
precisely the Julia directions for $f$ in a strong sense. Indeed,
for any $a\in\C$, $f$ has at most finitely many $a$-points outside
of the domains $\Lambda_1(\theta^*,c)$ by Theorem A.

Finally, it seems that the best possible error term in \eqref{zeros-n} might be $O\left(r^{q-p}\right)$, but our method does not give this.


\section{Asymptotic growth of $T(r,f)$ and $N(r,1/f)$}\label{main-results-sec}


For an exponential polynomial $f$ of the normalized form
\eqref{standard-f}, Steinmetz \cite{Stein1} has proved the asymptotic equations
	\begin{eqnarray}
	N\left(r,\frac{1}{f}\right)&=&C(\co(W_f))\frac{r^q}{2\pi}+o(r^q),\label{N}\\	
	T(r,f)&=&C(\co(W_f^0))\frac{r^q}{2\pi}+o(r^q).\label{T}
	\end{eqnarray}
These are higher order analogues of \eqref{Nq1} and \eqref{Tq1}.
The next result improves the error term in \eqref{T} even in the
standard case $p=1$. It is obvious that the logarithmic term in \eqref{T-eqn} below is only needed when $q=p$, that is, when the coefficients $H_j(z)$ are ordinary polynomials.

\begin{theorem}\label{T-theorem}
Let $f$ be an exponential polynomial in the normalized form \eqref{standard-f}, where
we suppose that $\rho(H_j)\leq q-p$ for $1\leq p\leq q$. Then
    \begin{equation}\label{T-eqn}
    T(r,f)=C(\co(W_f^0))\frac{r^q}{2\pi}+O\left(r^{q-p}+\log r\right).
    \end{equation}
\end{theorem}

Let $f$ be the exponential polynomial in \eqref{exp-poly}, and suppose
that $f$ can be written in the normalized form \eqref{standard-f},
where $\rho(H_j)\leq q-p$ for $1\leq p\leq q$. If one of the polynomials $Q_j(z)$ vanishes identically, we may suppose
that it is $Q_0(z)$, and that no other polynomial $Q_j(z)$ vanishes identically. In this case $H_0(z)$ reduces to an ordinary polynomial
in $z$. Under these conditions, we have the following result that
improves the error term in \eqref{N}.

\begin{theorem}\label{N-theorem}
If $H_0(z)\not\equiv 0$, we have
	\begin{equation}\label{m1}
	m\left(r,\frac{1}{f}\right)=O\left(r^{q-p}+\log r\right).
	\end{equation}
If $Q_0(z)\equiv 0$, then we have the stronger estimate
	\begin{equation}\label{m2}
	m\left(r,\frac{1}{f}\right)=O\left(\log r\right).
	\end{equation}
Moreover, we have
	\begin{equation}\label{N2}
	N\left(r,\frac{1}{f}\right)
	=C(\co(W_f))\frac{r^q}{2\pi}+O\left(r^{q-p}+\log r\right).
	\end{equation}
\end{theorem}

We conclude this section by the following quick consequence of \eqref{N}. This is a weaker form than \eqref{n-asy}, but it is needed in proving Theorem~\ref{length}, whereas \eqref{n-asy} is a consequence of Theorem~\ref{length}.

\begin{corollary}\label{n-cor}
Let $f$ be an exponential polynomial of order $q$ such that
$C=C(\co(W_f))>0$.  Then	
	\begin{equation*}
	n(r,1/f)=qC(\co(W_f))\frac{r^q}{2\pi}+o\left(r^{q}\right).
	\end{equation*}
\end{corollary}

\begin{proof}
Write \eqref{N} in the form
	$$
	C(\co(W_f))\frac{r^q}{2\pi}-\veps(r)r^q\leq N\left(r,\frac{1}{f}\right)\leq C(\co(W_f))\frac{r^q}{2\pi}+\veps(r)r^q,
	$$
where $0<\veps(r)<1$ is a decreasing function such that $\veps(r)\to 0$ as $r\to\infty$. Let $0<\delta(r)<1$ be any decreasing function such that $\delta(r)\to 0$ as $r\to\infty$. Using $\log x\geq \frac{x-1}{x}$ for $x\geq 1$, we get
	\begin{equation*}
	\begin{split}
	n(r)&=\frac{n(r)}{\log \left(1+\delta(r)\right)}
	\int_r^{r\left(1+\delta(r)\right)}\frac{dt}{t}
	\leq \frac{1+\delta(r)}{\delta(r)}
	\int_r^{r\left(1+\delta(r)\right)}\frac{n(t)}{t}\, dt\\
	&=\frac{1+\delta(r)}{\delta(r)}
	\big(N\left(r\left(1+\delta(r)\right),1/f\right)-N(r,1/f)\big)\\
	&=\frac{1+\delta(r)}{\delta(r)}
	\left(\frac{Cr^q}{2\pi}\sum_{j=1}^q{q\choose j}\delta(r)^{j}
	+O\left(\veps(r)r^{q}\right)\right)\\
	&=qC\frac{r^q}{2\pi}+O\left(\delta(r)r^q\right)+O\left(\veps(r)r^q\right)+O\left(\frac{\veps(r)r^q}{\delta(r)}\right).
	\end{split}
	\end{equation*}	
Similarly, using $\log x\leq x-1$ for $x\geq 1$, we get	
	\begin{equation*}
	\begin{split}	
	n(r)&=\frac{n(r)}{-\log \left(1-\delta(r)\right)}
	\int_{r\left(1-\delta(r)\right)}^r\frac{dt}{t}
	\geq \frac{1-\delta(r)}{\delta(r)}\int_{r\left(1-\delta(r)\right)}^r
	\frac{n(t)}{t}\, dt\\
	&=qC\frac{r^q}{2\pi}+O\left(\delta(r)r^q\right)+O\left(\veps(r)r^q\right)+O\left(\frac{\veps(r)r^{q}}{\delta(r)}\right).
	\end{split}
	\end{equation*}
Choose $\delta(r)=\sqrt{\veps(r)}$, and all error terms above are of the form $o\left(r^q\right)$.
\end{proof}


\section{Lemmas for Theorem~\ref{T-theorem}}\label{lemmas-T-sec}


Let $f$ be an exponential polynomial with no zeros. Then $f=e^P$, where $P$ is a polynomial of degree $p$,
and there exist constants $C>0$ and $R>0$ such that
	\begin{equation}\label{eP}
	\log |f(z)|\geq \log e^{-|P(z)|}\geq -Cr^p
	\end{equation}
for $|z|>R$. If $f$ has finitely many zeros, then $f=Qe^P$, where $Q$ is also a polynomial. The conclusion
in \eqref{eP} holds in this case also, but possibly for different constants $C$ and $R$.

For an exponential polynomial $f$ with infinitely many zeros, a lower bound for $\log |f(z)|$
is given in  \cite[Lemma 5.1]{HITW}, which originates from \cite[Theorem~V.~19]{Tsuji}.
This estimate is valid for all $z$ outside of certain discs centered at the zeros of $f$.
In our applications, however, we need to be more flexible with the size of these discs, as is
described next.

\begin{lemma}\label{minmods-lemma}
Let $f$ be an entire function of order $\rho=\rho(f)$, let $\lambda=\lambda(f)$ denote the exponent of convergence of zeros of $f$, and
let $k:(0,\infty)\to [1,\infty)$ be any non-decreasing function.
Suppose that $\lambda\geq 1$ and that the number of zeros
of $f$ in $|z|<r$ satisfies
    $$
    n(r,1/f)=C(r)r^\lambda,
    $$
where $C(r)$ is bounded as $r\to\infty$. Then there exists a constant $A>0$ such that
	\begin{equation}\label{mod-f0}
	\log |f(z)|\geq -A\max\left\{r^{\rho},r^{\lambda}\log \left(rk(2r)\right)\right\}
	\end{equation}
whenever $z=re^{i\theta}$ lies outside of the discs
$|z-z_n|\leq 1/k(|z_n|)$ and $|z|\leq e$.
\end{lemma}

\begin{proof}
We may factorize the entire $f$ as $f=Qe^P$, where $P$ is a polynomial and $Q$ is a canonical
product formed with the zeros $z_n$ of $f$.
Set $p=\deg(P)$, in which case $\rho=\max\{p,\lambda\}$, and
$q=\lfloor\lambda\rfloor$.
Since the integral $\int_1^\infty\frac{N(r,1/f)}{r^{q+1}}\, dr$ diverges, the sum
$\sum_{n=1}^\infty |z_n|^{-q}$ also diverges, see \cite[Lemma 1.4]{Hayman}. Hence $q$ is the
genus of the sequence $\{z_n\}$, see \cite[p.~216]{Tsuji} for the definition. This allows us to
write $Q$ in the form
	$$
	Q(z)=\prod_{n=1}^\infty\left(1-\frac{z}{z_n}\right)\exp\left(\frac{z}{z_n}+\cdots+\frac{1}{q}
	\left(\frac{z}{z_n}\right)^q\right).
	$$
We also write
	$$
	 \Psi(z,z_n)=\log\frac{z_n}{z_n-z}-\left(\frac{z}{z_n}+\cdots+\frac{1}{q}\left(\frac{z}{z_n}\right)^q\right),
	$$
so that
	$$
	\log\frac{1}{Q(z)}=\sum_{n=1}^\infty \Psi(z,z_n),\quad z\neq z_n.
	$$
Taking positive real parts on both sides of the above equation results in
	$$
	\log^+\frac{1}{|Q(z)|}\leq \sum_{r_n\leq 2r}\Re^+(\Psi(z,z_n))
	+\sum_{r_n>2r}|\Psi(z,z_n)|=:\Sigma_1+\Sigma_2,
	$$
where $r_n=|z_n|$. Similarly as in the proof of  \cite[Lemma 5.1]{HITW}, we obtain $\Sigma_2=O\left(r^q\right)$. This estimate has nothing to do with the size of the discs around the points $z_n$.
In $\Sigma_1$, we have
	\begin{eqnarray*}
	\Re^+(\Psi(z,z_n))&\leq & \log^+\left|\frac{z_n}{z-z_n}\right|+\sum_{j=1}^q\left(\frac{r}{r_n}\right)^j\\
	&\leq & \log^+\left|\frac{z_n}{z-z_n}\right|+\left(\frac{r}{r_n}\right)^q\sum_{j=0}^{q-1}\left(\frac{r_n}{r}\right)^{j}\\
	&\leq & \log^+\left|\frac{z_n}{z-z_n}\right|+2^q\left(\frac{r}{r_n}\right)^q.
	\end{eqnarray*}	
If $z$ lies outside of the discs $|z-z_n|\leq 1/k(|z_n|)$ and $|z|\leq e$, it follows that
	$$
	\log^+\left|\frac{z_n}{z-z_n}\right|\leq \log\left(r_nk(r_n)\right)\leq\log \left(2rk(2r)\right),
	$$
because $r_n\leq 2r$ and $k$ is non-decreasing. Thus
	\begin{eqnarray*}
	\Sigma_1&\leq& n(2r)\log \left(2rk(2r)\right)+2^qr^q\int_e^{2r}\frac{dn(t)}{t^q}+O(1)\\
	&\leq& O\left(r^\lambda\log \left(rk(2r)\right)\right)+2^qqr^q\int_e^{2r}\frac{n(t)}{t^{q+1}}\, dt
	=O\left(r^\lambda\log \left(rk(2r)\right)\right).
	\end{eqnarray*}
	
The reasoning above shows that there exists a constant $B>0$ such that
$\log |Q(z)|\geq -Br^\lambda\log \left(rk(2r)\right)$
whenever $z=re^{i\theta}$ lies outside of the aforementioned closed discs. Since $P$ is a polynomial of
degree $p$, there is a constant $C>0$ such that $\log \left|e^{P(z)}\right|\geq -Cr^p$ for $|z|>e$.
The assertion now follows by choosing $A=B+C$.
\end{proof}

\begin{corollary}\label{minmod-cor}
Let $f$ be an exponential polynomial of order $\rho=\rho(f)$, and let $\lambda=\lambda(f)\geq 1$ be the exponent of convergence of zeros of $f$. Suppose that $k:(0,\infty)\to [1,\infty)$ is any non-decreasing function.
Then there exists a constant $A>0$ such that
	$$
	\log |f(z)|\geq
	-A\max\left\{r^{\rho},r^\lambda\log \left(rk(2r)\right)\right\}
	$$
whenever $z=re^{i\theta}$ lies outside of the discs
$|z-z_n|\leq 1/k(|z_n|)$ and $|z|\leq e$.
\end{corollary}

\begin{remark}
We may choose $k(x)=x^{\lambda+\veps}$, $\veps>0$, in Lemma~\ref{minmods-lemma} and $k(x)=x^q\log^2x$ in Corollary~\ref{minmod-cor}. For these
choices of $k$, it follows by Riemann-Stieltjes integration that
	\begin{equation*}
	\sum_{|z_n|\geq e}\frac{1}{k(|z_n|)}=\int_e^\infty\frac{dn(t)}{k(t)}<\infty.
	\end{equation*}
Thus the projection $I$ of the discs $|z-z_n|\leq 1/k(|z_n|)$
on the positive real axis has finite linear measure.
\end{remark}

The following lemma follows implicitly from the discussions in
\cite[Section~3]{HN}. For the convenience of the reader, we give
a direct proof.

\begin{lemma}\label{T-lemma}
Suppose that $T:\R_+\to\R_+$ is eventually a non-decreasing function.
Suppose further that there exist constants $C>0$, $q>0$, $0\leq p<q$
and $s\geq 0$ such that,
as $r\to\infty$, $T$ can be written in the form
	\begin{equation}\label{german-equation}
	T(r)=Cr^q+O\left(r^p\log^s r\right),\quad r\not\in E\cup [0,1],
	\end{equation}
where $E\subset [1,\infty)$ has finite logarithmic measure (or finite linear measure). Then
the asymptotic equality in \eqref{german-equation} holds for all
$r$ large enough.
\end{lemma}

\begin{proof}
Let $\veps\in (0,\min\{1,C\})$. Then there exist constants $R_1=R_1(\veps)>1$
and $C_1>0$ such that $T(r)$ is increasing for $r>R_1$ and
	$$
	(C-\veps)r^q-C_1r^p\log^s r\leq T(r)
	\leq (C+\veps)r^q+C_1r^p\log^s r,
	$$
where $r\in (R_1,\infty)\setminus E$. Moreover, due to $p<q$ we may
suppose that $R_1$ is chosen large enough so that the lower
bound and the upper bound for $T(r)$ are increasing functions of $r$.
Then, using \cite[Lemma~5]{Gundersen2}, we can find an
$R_2=R_2(\veps)>\max\{2,R_1\}$ such that, for $r>R_2$,
	\begin{eqnarray*}
	T(r) &\leq& (C+\veps)(1+\veps)^qr^q+C_1(2r)^p\log^s (2r)\\
	&\leq& (C+\veps)(1+\veps)^qr^q+2^{p+s}C_1r^p\log^s r,
	\end{eqnarray*}
and similarly
	$$
	T(r) \geq \frac{C-\veps}{(1+\veps)^q}r^q-C_1r^p\log^s r.
	$$
In other words,
	$$
	(C-\veps')r^q-C_2r^p\log^s r\leq T(r)
	\leq (C+\veps')r^q+C_2r^p\log^s r,\quad r>R_2,
	$$
where $R_2=R_2(\veps')>2$ and $C_2=2^{p+s}C_1$.
This yields the assertion.
\end{proof}

\begin{lemma}[\cite{HITW}]\label{indicator-lemma}
Let $f$ be an exponential polynomial of the form \eqref{standard-f},  $\psi_k(\theta)=\Re\left(w_ke^{iq\theta}\right)$, $p\leq q$ and
	\begin{equation}\label{2a}
	2a=\min\left\{|\psi_k'(\theta)-\psi_j'(\theta)| : \psi_k(\theta)=\psi_j(\theta),\ k\neq j\right\}.
	\end{equation}
For a given $c>0$ there exists an $R_0>0$, with the following property: If $z=re^{i\theta}\not\in \Lambda_p(\theta^*,c)$
for any $\theta^*$ and any $r\geq R_0$, and if $h_f(\theta)=\psi_j(\theta)$, then
	$$
	h_f(\theta)>\max_{k\neq j}\psi_k(\theta)+ac\frac{\log r}{r^p}.
	$$
\end{lemma}


\section{Proof of Theorem~\ref{T-theorem}}\label{proof-T-sec}


We divide the proof into four steps for clarity.

\medskip
\noindent
\textbf{5.1.~Preliminaries.}
Let $\arg(z)=\theta_j$, $j=1,\ldots,k$, be the critical rays of $f$ organized such that $0\leq \theta_1<\theta_2<\cdots<\theta_k<2\pi$. Set $\theta_{k+1}=\theta_1+2\pi$.
For $j=1,\ldots,k$, define $F_j=(\theta_{j},\theta_{j+1})$ and
	\begin{eqnarray*}
	\Lambda_p(\theta_j,c)&=&\left\{z=re^{i\theta} : r>1,\ |\arg(z)-\theta_j|
	<c\frac{\log r}	{r^p} \right\},\\
	F_j(r)&=&\left(\theta_j+c\frac{\log r}{r^p},\theta_{j+1}-c\frac{\log r}{r^p}\right),
	\end{eqnarray*}
where $c>0$ is sufficiently large. The domain $\Lambda_p(\theta_j,c)$ curves
asymptotically towards the critical ray $\arg(z)=\theta_j$, provided that $p\geq 2$.
The interval $F_j(r)$ corresponds to the arguments in between two consequtive
domains $\Lambda_p(\theta_j,c)$ and $\Lambda_p(\theta_{j+1},c)$.
Clearly $F_j(r)\to (\theta_j,\theta_{j+1})=F_j$ as $r\to\infty$ for any $p\geq 1$.

Since $f$ has no poles, we have
	$$
	T(r,f)=m(r,f)=\frac{1}{2\pi}\int_0^{2\pi}\log^+|f(re^{i\theta})|\, d\theta.
	$$
For $j=0,\ldots,m$, we define	
	\begin{eqnarray*}
	E_j &=& \left\{\theta\in[0,2\pi) :
	\Re\left(w_je^{iq\theta}\right)>\max_{k\neq j}\Re\left(w_ke^{iq\theta}\right)\right\},\\
	E_j(r) &=& \left\{\theta\in[0,2\pi) :
	\Re\left(w_je^{iq\theta}\right)> \max_{k\neq j}\Re\left(w_ke^{iq\theta}\right)
	+c\frac{\log r}{r^p}\right\}.
	\end{eqnarray*}
By Lemma~\ref{indicator-lemma}, each set $E_j(r)$ is a finite union
of intervals $F_i(r)$. Thus set $E_j$ is a finite union of intervals $F_i$.
If $E_0\neq\emptyset$, then $P_0(z)\not\equiv 0$ and $w_0=0$, so the definition of the set $E_j$ implies that the terms $|H_j(z)e^{w_jz^q}|$, $j=1,\ldots,m$, are bounded for
$\arg(z)=\theta\in E_0$. On the other hand, if $\theta\not\in E_j$ for some $j$, then $|H_j(z)e^{w_jz^q}|\leq M(r,H_j)$. Keeping in mind that $\rho(H_j)\leq q-p$, we conclude in both cases $E_0=\emptyset$ and $E_0\neq\emptyset$ that
	\begin{equation}\label{T-sum}
	T(r,f)=\sum_{j=1}^m\frac{1}{2\pi}\int_{E_j}\log^+|f(re^{i\theta})|\, d\theta
    +O\left(r^{q-p}+\log r\right).
	\end{equation}

\medskip
\noindent
\textbf{5.2.~Upper bound for $T(r,f)$.}
To estimate the integrals in \eqref{T-sum} over the sets $E_j$, we write
	\begin{equation}\label{one-term-out}
	f(z)=e^{w_jz^q}\left(H_j(z)+\sum_{k\neq j} H_k(z)e^{(w_k-w_j)z^q}\right).
	\end{equation}
If $z=re^{i\theta}$ and $\theta\in E_j$, it follows that
	\begin{equation}\label{f-up}
	\begin{split}
    \log^+|f(z)| &\leq \Re\left(w_je^{iq\theta}\right)+\sum_{j=0}^m \log^+M(r,H_j)+\log (m+1)\\
	&= k_f(q\theta)r^q +O\left(r^{q-p}+\log r\right),
    \end{split}
	\end{equation}
where $k_f(\theta)$ is the support function for the convex set $\co(W_f^0)$
\cite[p.~74]{Levin1}. Therefore
	\begin{eqnarray*}
    T(r,f)
    &\leq&\sum_{j=1}^m \frac{r^q}{2\pi}\int_{E_j}k_f(q\theta)\, d\theta
    +O\left(r^{q-p}+\log r\right)\\
    &\leq& \frac{r^q}{2\pi}\int_0^{2\pi}k_f(q\theta)\, d\theta
    +O\left(r^{q-p}+\log r\right)\\
    &=& \frac{r^q}{2\pi}\int_0^{2\pi}k_f(\varphi)\, d\varphi
    +O\left(r^{q-p}+\log r\right)\\
    &=&C(\co(W_f^0))\frac{r^q}{2\pi}+O\left(r^{q-p}+\log r\right),
    \end{eqnarray*}
where the last identity follows by Cauchy's formula for convex
curves. This formula can be found in many books on integral geometry.

\medskip
\noindent
\textbf{5.3.~Lower bound for $T(r,f)$.}
If $q=1$ or if $p=q$, then the coefficients $H_j(z)$ are polynomials, and no exceptional
set other than $|z|$ is large enough occurs when estimating $T(r,f)$ downwards.
Hence we suppose that $1\leq p\leq q-1$. Observe that
	$$
	\log^+(xy)\geq \max\left\{\log^+x-\log^+\frac{1}{y},0\right\}
	$$
for all $x\geq 0$ and $y>0$. So, if $z=re^{i\theta}$ and $\theta\in E_j(r)$, by applying
this observation to \eqref{one-term-out}, we have
	\begin{eqnarray*}
	\log^+|f(z)| &\geq& \max\left\{\Re\left(w_je^{iq\theta}\right)
	-\log^+\left||H_j(z)|-\sum_{k\neq j} M(r,H_k)
	\left|e^{(w_k-w_j)z^q}\right|\right|^{-1},0\right\}\\
	&\geq& \max\left\{h_f(\theta)r^q -\log^+\left||H_j(z)|-\sum_{k\neq j}
	M(r,H_k)e^{-cr^{q-p}\log r}\right|^{-1},0\right\}.
	\end{eqnarray*}
The functions $H_j(z)$ are of order $\leq q-p\leq q-1$ and of finite type. Therefore there
exists a constant $A>0$ such that
    $$
    \log M(r,H_j)\leq Ar^{q-p}\log r\quad\textnormal{and}\quad
    \log |H_j(z)|\geq -Ar^{q-p}\log r.
    $$
The latter inequality is valid for all $|z|=r$ outside of a set $I\subset (0,\infty)$
of finite linear measure. This follows by applying Lemma~\ref{minmods-lemma} with
$k(x)=x^q$ to each of the coefficients $H_j(z)$, and then taking the union
of all the exceptional sets involved. See also the remark following
Lemma~\ref{minmods-lemma}. We note that $k(x)$ could be chosen to have less
growth at this point, but the particular growth rate fixed here is needed later on.

Choosing $c>2a$ in Lemma \ref{indicator-lemma}, we have
    \begin{equation}\label{f-down}
    \log^+ |f(z)|\geq \max\{h_f(\theta)r^q,0\}-O\left(r^{q-p}+\log r\right),
    \end{equation}
where $|z|=r\not\in I$. Recall that each set $E_j$ is a finite union of intervals $F_i$.
In fact, $E(r):=\cup_jE_j(r)=\cup_iF_i(r)$, so that
	\begin{eqnarray*}
	\int_{[0,2\pi]\setminus E(r)}\log^+|f(re^{i\theta})|\, d\theta
	&=&\sum_{j=1}^n\int_{\theta_j+c\frac{\log r}{r^p}}^{\theta_{j+1}-c\frac{\log r}{r^p}}
	\log^+|f(re^{i\theta})|\, d\theta\\
    &=&O\left(r^{q-p}+\log r\right).
	\end{eqnarray*}
From \eqref{T-sum}, \eqref{f-up} and \eqref{f-down}, it then follows that
    \begin{eqnarray*}
    T(r,f) &\geq& \sum_{j=1}^m \frac{r^q}{2\pi}\int_{E_j(r)}\max\{h_f(\theta),0\}\, d\theta
    +O\left(r^{q-p}+\log r\right)\\
    &=&\sum_{j=1}^m \frac{r^q}{2\pi}\int_{E_j}\max\{h_f(\theta),0\}\, d\theta
    +O\left(r^{q-p}+\log r\right)\\
    &=& \frac{r^q}{2\pi}\int_0^{2\pi}\max\{h_f(\theta),0\}\, d\theta
    +O\left(r^{q-p}+\log r\right)\\
    &=&C(\co (W_f^0))\frac{r^q}{2\pi}+O\left(r^{q-p}+\log r\right),\quad r\not\in I,
    \end{eqnarray*}
where the last identity follows again by Cauchy's formula and by the
fact that $h_f(\theta)=k_f(q\theta)$.

\medskip
\noindent
\textbf{5.4.~Conclusion of the proof.}
The upper and lower estimates just obtained for $T(r,f)$ have the same magnitude
of growth, but the lower estimate is valid outside of an exceptional set $I$ of finite
linear measure. The set $I$ can be avoided by means of Lemma~\ref{T-lemma}.
This completes the proof of \eqref{T-eqn}.


\section{Proof of Therem~\ref{N-theorem}}\label{proof-N-sec}


Let us begin with the particular case $Q_0(z)\equiv 0$. As  the
functions $g_j(z)=P_j(z)e^{Q_j(z)}$, $j=1,\ldots,n$, are linearly independent, they
form a fundamental solution base for a linear differential equation
	\begin{equation}\label{lde}
	L_{n}(g):=g^{(n)}+a_{n-1}(z)g^{(n-1)}+\cdots+a_0(z)g=0,
	\end{equation}
see \cite[Proposition~1.4.6]{Laine}. The coefficient functions $a_j(z)$ can be expressed
as quotients of Wronskian and modified Wronskian determinants of the functions $g_j(z)$.
Note that $g_j'(z)=R_j(z)e^{Q_j(z)}$, where
	$$
	R_j(z)=P_j'(z)+Q_j'(z)P_j(z)\not\equiv 0.
	$$
This means that the exponential term $e^{Q_j(z)}$ remains in differentiation, and
further differentiations will not change that. Therefore, from each column of each
determinant we can pull out the single exponential term as
a common factor. Thus $\exp\big(Q_1(z)+\cdots+Q_n(z)\big)$ is a common factor for each
determinant, and they cancel out in the quotients. The entries of the remaining determinants
are polynomials. Hence the coefficients $a_j(z)$ in \eqref{lde} are rational functions.

If $P_0(z)$ would be a solution of \eqref{lde}, it would have to be  linearly dependent
with the solutions $g_1,\ldots,g_n$. Thus $P_0(z)$ would be transcendental, which is
a contradiction. It follows that $a_n(z):=L_{n}(P_0)\not\equiv 0$ is a rational function. We conclude that $f$
is a solution of the non-homogeneous equation
	$$
	f^{(n)}+a_{n-1}(z)f^{(n-1)}+\cdots+a_0(z)f=a_n(z).
	$$
Dividing both sides by $fa_n(z)$ and using the lemma on the logarithmic derivative
together with the fact that the coefficients are rational gives us \eqref{m2}.

\medskip
Suppose next that $Q_0(z)\not\equiv 0$.	Define
	\begin{equation}\label{g}
	g(z)=e^{-Q_0(z)}f(z)=P_0(z)+\cdots+P_n(z)e^{Q_n(z)-Q_0(z)}.
	\end{equation}
Now the discussion above can be applied to $g$ giving us $m(r,1/g)=O(\log r)$
and, a fortiori,
	\begin{equation}\label{N1}
	N\left(r,\frac{1}{f}\right)=N\left(r,\frac{1}{g}\right)=T(r,g)+O(\log r).
	\end{equation}
Suppose in addition that $H_0(z)\not\equiv 0$. Due to the assumption $\rho(H_j)\leq q-p$, we may assume, without loss of generality, that $\deg(Q_0)\leq q-p$. Thus, using
Theorem~\ref{T-theorem} to $f$ and using \eqref{g}, we get
	\begin{equation}\label{T1}
	\begin{split}
	T(r,f)&=C(\co(W_f^0))\frac{r^q}{2\pi}+O\left(r^{q-p}+\log r\right)\\
	&=T(r,g)+O\left(r^{q-p}+\log r\right).
	\end{split}
	\end{equation}
We deduce the estimate \eqref{m1} by combining \eqref{N1} and \eqref{T1}.
Finally, we suppose that $H_0(z)\equiv 0$, in which case $\deg(Q_j)=q$ for every $j$ and $m=n$. Let $U=\{\overline{w}_1-\overline{w}_0,\ldots,\overline{w}_m-\overline{w}_0\}$ be the set of conjugate leading coefficients of $g$, and set $U_0=U\cup\{0\}$. By a simple vector
calculus, we conclude the following: If $\overline{w}_0$ is a boundary point of $\co(W_f)$,
then $0$ is a boundary point of $\co(U)$, while if $\overline{w}_0$ is an interior point of
$\co(W_f)$, then $0$ is an interior point of $\co(U)$. Thus, applying Theorem~\ref{T-theorem}
to $g$, we get
	\begin{eqnarray*}
	T(r,g) &=& C(\co(U_0))\frac{r^q}{2\pi}+O\left(r^{q-p}+\log r\right)\\
	&=& C(\co(W_f))\frac{r^q}{2\pi}+O\left(r^{q-p}+\log r\right).
	\end{eqnarray*}
The final assertion \eqref{N2} follows from this and \eqref{N1}.


\section{Lemmas for Theorem~\ref{length}}\label{lemmas-Polya}


The following lemma is considered for more general curves in \cite{DHW} but in Cartesian coordinates. However, our situation is very delicate, and we need the representation particularly in polar coordinates.

\begin{lemma}\label{meet-lemma}
Let $p\geq 1$ be any real number, and let $U$ be any collection of Euclidean discs $D_n=D(z_n,r_n)$, where the center points $z_n\in\C$ are ordered according to increasing moduli, $|z_n|\to\infty$, $r_n>0$,
$r_n\to 0$ and
	\begin{equation}\label{convergence}
	\sum_{|z_n|>e} \frac{|z_n|^{p-1}r_n}{\log |z_n|}<\infty.
	\end{equation}
Then the set $C\subset (0,\infty)$ of values $c$ for which the curve
    $$
	\partial\Lambda_p(0,c)=\left\{z=re^{i\theta} : r\geq 1,\ |\arg(z)|=c\frac{\log r}{r^p} \right\}
	$$
meets infinitely many discs $D_n$ has linear measure zero.
\end{lemma}

\begin{proof}
If $z\in \partial\Lambda(0,c)$, then
	$$
	z=r\cos\left(c\frac{\log r}{r^p}\right)\pm ir\sin\left(\frac{\log r}{r^p}\right)\sim r\pm ic\frac{\log r}{r^p}.
	$$
Asymptotically this corresponds to the Cartesian curves $y=\pm c\frac{\log x}{x^{p-1}}$. From this representation we see that the cases $p=1$ and $p>1$ have a different geometry. Nevertheless our proof works for both cases simultaneously.

By the Cartesian representation we may suppose that the points $z_n$ lie in the right half-plane in between the lines $y=\pm x$. Since $|z_n|\to\infty$ and $r_n\to 0$, there exists a positive integer $N$ such that
	$$
	|z_n|-r_n\geq |z_n|/2\geq e,\quad n\geq N.
	$$
Note that the function $x\mapsto \frac{\log x}{x^p}$ is strictly decreasing for $x\geq e$. Moreover, it suffices to consider the portion of $\partial\Lambda_p(0,c)$ located in the upper half-plane, call it $\Lambda^+(c)$ for short.

Suppose that a disc $D_n$ lies asymptotically in between the curves $\Lambda^+(c_1)$ and $\Lambda^+(c_2)$, where $c_1<c_2$, and let $\zeta_1$ and $\zeta_2$ denote the respective intersection points. The disc $D_n$
can be seen from the origin at an angle $2\theta_n$, where $\sin\theta_n=r_n/|z_n|$.

 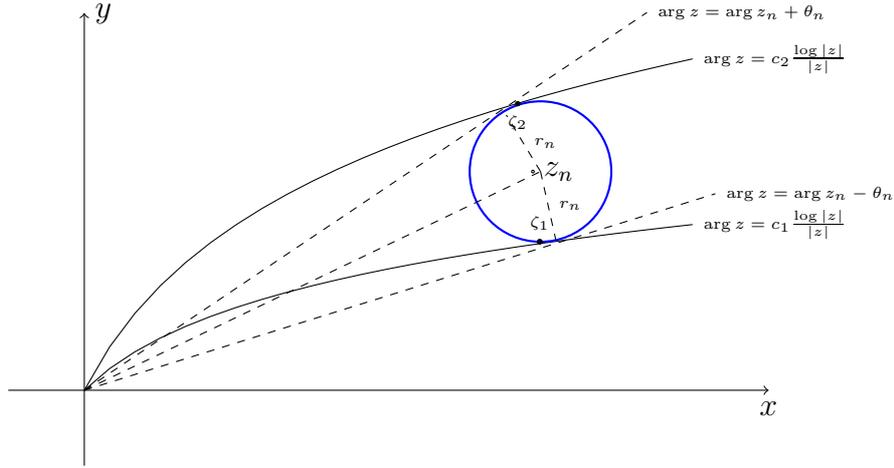
\begin{figure}[h]\label{measure}
    \begin{center}
    \begin{tikzpicture}
    \draw[->](0,0)--(10,0)node[left,below]{$x$};
    \draw[->](1,-1)--(1,5)node[right]{$y$};
    \draw[domain=1:9] plot(\x,{ln(\x)})node[right,font=\tiny]{$\arg z =c_1\frac{\log|z|}{|z|}$};
     \draw[domain=1:9] plot(\x,{2*ln(\x)})node[right,font=\tiny]{$\arg z =c_2\frac{\log|z|}{|z|}$};
    \draw[thick,blue](7,2.9) circle [radius=26.5pt];
    \draw[thin](6.9,2.9) circle [radius=0.5pt]node[right]{$z_n$};
    \draw[thick](6.7,3.8) circle [radius=0.6pt]node[below=0.01pt,font=\tiny]{$\zeta_2$};
    \draw[thick](6.99,1.97) circle [radius=0.7pt]node[above=0.1pt,font=\tiny]{$\zeta_1$};
    \draw[-,dashed](1,0)--(7,2.9);
     \draw[-,dashed](7,2.9) to node[right,font=\tiny]{$r_n$} (6.55,3.65);
      \draw[-,dashed](7,2.9) to node[right,font=\tiny]{$r_n$} (7.2,2);
     \draw[-,dashed][domain=1:8.4] plot(\x,{0.677*(\x-1)})node[above,right,font=\tiny]{$\arg z=\arg z_n+\theta_n$};
     \draw[-,dashed][domain=1:9.3] plot(\x,{0.314*(\x-1)})node[above,right,font=\tiny]{$\arg z=\arg z_n-\theta_n$};
    \end{tikzpicture}
    \end{center}
	\begin{quote}
    \caption{The disc $D_n$ and the asymptotic curves in the case $p=1$.}
    \end{quote}
    \end{figure}

We have $\theta_n\leq 2r_n/|z_n|\to 0$ as $n\to\infty$, and
	\begin{eqnarray*}
	c_2-c_1&=&\frac{|\zeta_2|^p\arg(\zeta_2)}{\log |\zeta_2|}-\frac{|\zeta_1|^p\arg(\zeta_1)}{\log |\zeta_1|}\\
	&\leq& \frac{(|z_n|+r_n)^p\arg(\zeta_2)}{\log (|z_n|+r_n)}-\frac{(|z_n|-r_n)^p\arg(\zeta_1)}{\log (|z_n|-r_n)}\\
	&\leq & \frac{|z_n|^p\left(1+\frac{r_n}{|z_n|}\right)^p(\arg(z_n)+\theta_n)}{\log (|z_n|+r_n)}-\frac{|z_n|^p\left(1-\frac{r_n}{|z_n|}\right)^p(\arg(z_n)-\theta_n)}{\log (|z_n|-r_n)}\\
	&\leq & \frac{|z_n|^p\arg(z_n)+2|z_n|^p\theta_n}{\log (|z_n|-r_n)}\left(\left(1+\frac{r_n}{|z_n|}\right)^p-\left(1-\frac{r_n}{|z_n|}\right)^p\right).
	\end{eqnarray*}
Using the estimate
	\begin{eqnarray*}
	\left(1+\frac{r_n}{|z_n|}\right)^p-\left(1-\frac{r_n}{|z_n|}\right)^p	
	&\leq& 2^p\left(1-\left(\frac{|z_n|-r_n}{|z_n|+r_n}\right)^p\right)\\
	&=& 2^pp\int_\frac{|z_n|-r_n}{|z_n|+r_n}^1x^{p-1}\, dx\leq \frac{2^{p+1}pr_n}{|z_n|},
	\end{eqnarray*}
we conclude that
	\begin{eqnarray*}
	c_2-c_1 &\leq & \frac{2^{p+1}p|z_n|^{p-1}r_n(\arg(z_n)+2\theta_n)}{\log (|z_n|-r_n)}\leq \frac{2^{p}p\pi |z_n|^{p-1}r_n}{\log (|z_n|/2)}.
	\end{eqnarray*}
	
Let $\varepsilon>0$. Then there exists an $N(\varepsilon)\in\N$ such that
$\sum_{n=N(\varepsilon)}^\infty \frac{2^{p}p\pi |z_n|^{p-1}r_n}{\log (|z_n|/2)}<\varepsilon$. Let $C_\varepsilon\subset (0,\infty)$ 
denote the set of
values $c$ such that the curve $\Lambda^+(c)$ meets at least one of
the discs $D_n$, where $n\geq N(\varepsilon)$. Then $C_\varepsilon$ has
linear measure $<\varepsilon$. Since $C$ is contained in all the sets
$C_\varepsilon$, $\varepsilon>0$, it follows that $C$ has linear measure
zero.	
\end{proof}

Through the rest of this section, let $k:(0,\infty)\to [1,\infty)$ denote any non-decreasing function. To simplify the notation, we restrict to the growth rate $k(r)=O\left(r^\sigma\right)$ for some $\sigma>0$, even though the results that will follow would allow even faster growth. Let $f$ be an exponential polynomial in the normalized form \eqref{standard-f}, and let $\{z_n\}$ denote the sequence of zeros of $f$ and of all of its transcendental coefficients $H_j$, listed according to multiplicities and ordered according to increasing moduli. Finally, let $\Ce(f,k)$ denote the collection of discs $|z-z_n|\leq 1/k(|z_n|)$. For simplicity, we may also assume that a disc $|z|<R$ for a suitably large $R>0$ is included in  $\Ce(f,k)$.

The next auxiliary result is obtained by modifying the proof of \cite[Theorem~2.4]{HITW}, see also the remark following \cite[Theorem~2.4]{HITW}.

\begin{lemma}\label{one-leading-term}
Let $f$ be an exponential polynomial in the normalized form \eqref{standard-f}, where
we suppose that $\rho(H_j)\leq q-p$ for $1\leq p\leq q$.
Then there exist constants $A>q$ and $c>0$ with the following properties:
If $\Lambda$ denotes the domain in between any two consecutive zero-domains
$\Lambda_p(\theta^*_1,c)$ and $\Lambda_p(\theta^*_2,c)$ of $f$, and if $z\in\Lambda\setminus\Ce(f,k)$,
then there exists a unique index $j_0$ such that
	\begin{equation}\label{f}
	f(z) = H_{j_0}(z)e^{w_{j_0}z^q}(1+\veps_{q-p}(r)),
    \end{equation}
where $|\veps_k(r)|\leq B\exp\left(-Ar^{k}\log r\right)$, $k\geq 0$,
$|z|=r$, and $B>0$ is some constant.
Denote
    $$
    G_{j}(z)=H'_{j}(z)+qw_{j}z^{q-1}H_{j}(z).
    $$
For the same values of $z$ and $j_0$ as above, we have
    \begin{equation}\label{f-prime}
	f'(z) =\left\{ \begin{array}{ll}
    G_{j_0}(z)e^{w_{j_0}z^q}(1+\veps_{q-p}(r)), & \ \textnormal{if $j_0\neq 0$},\\
    H_0'(z)+\veps_{q-p}(r), & \ \textnormal{if $j_0=0$}.
    \end{array}\right.
    \end{equation}
\end{lemma}

\begin{proof}
(1) In order to prove \eqref{f}, we first
suppose that $q\geq 2$ and that $1\leq p\leq q-1$. The latter means that all coefficients $H_j(z)$
are assumed to be transcendental. There exist constants $A>q$ and $r_0>0$ such that
	\begin{equation}\label{up}
	\log M(r,H_j)\leq Ar^{q-p},\quad j\in\{0,\ldots,m\},
	\end{equation}
whenever $r\geq r_0$. We conclude by Corollary~\ref{minmod-cor} that there exists a constant $A>q$ such that
	\begin{equation}\label{down}
	\log |H_j(z)|\geq -Ar^{q-p}\log r,\quad j\in\{0,\ldots,m\},
	\end{equation}
whenever $z\not\in\Ce(f,k)$. We may suppose that the constant $A$ in \eqref{down}
is the same as that in \eqref{up} by choosing the larger of the two.
Noting that $r^q\psi_k(\theta)=\Re\left(w_kz^q\right)$, we make use of Lemma~\ref{indicator-lemma} by choosing
$c>\frac{4A}{a}$: If $z=re^{i\theta}\in\Lambda$ and if $r$
is large enough, then there exists a unique index $j_0$ such that
    \begin{equation}\label{inequality}
    \Re \left(w_{j_0}z^q\right)\geq \Re\left(w_kz^q\right)+4Ar^{q-p}\log r
    \end{equation}
for all $k\neq j_0$. If in addition $z\not\in\Ce(f,k)$, then the estimates
\eqref{up}--\eqref{inequality} applied to \eqref{standard-f} yield
    \begin{eqnarray*}
    \left|f(z)e^{-w_{j_0}z^q}\right| &\geq&
    |H_{j_0}(z)|-\sum_{k\neq j_0}|H_k(z)||e^{(w_k-w_{j_0})z^q}|\\
    &\geq&
    |H_{j_0}(z)|-\exp\left(\log m+Ar^{q-p}-4Ar^{q-p}\log r\right),
	\end{eqnarray*}
from which
	$$
	\frac{|f(z)|}{\left|H_{j_0}(z)e^{w_{j_0}z^q}\right|}\geq 1-\exp\left(-Ar^{q-p}\log r\right).
	$$
On the other hand,
	$$
	\frac{|f(z)|}{\left|H_{j_0}(z)e^{w_{j_0}z^q}\right|}
	\leq 1+\sum_{k\neq j_0}\frac{|H_k(z)|}{|H_{j_0}(z)|}\left|e^{(w_k-w_{j_0})z^q}\right|
	\leq 1+\exp\left(-Ar^{q-p}\log r\right),
	$$
where $z\in\Lambda\setminus\Ce(f,k)$. This discussion covers the case $w_{j_0}=w_0=0$ also, that is,
the case when $f(z)-H_0(z)=\veps_{p-q}(r)$ in $\lambda\setminus\Ce(f,k)$.

If some (but not all) coefficients are polynomials, the previous reasoning simplifies. Indeed, if a particular
coefficient $H_j(z)$ is a polynomial, then the growth of $|H_j(z)|$ is comparable to $|z|^{\deg(H_j)}$
for $r$ large enough. Consideration of the set $\Ce(f,k)$ is not needed for this particular coefficient,
as we only need to assume that $|z|$ is large enough.

Suppose then that $q\geq 1$ is any integer and $p=q$, that is, all of the coefficients $H_j(z)$ are polynomials.
This covers the remaining case. There exist constants $A_1>0$ and $A_2>0$ such that
	$$
	|H_j(z)|\leq A_1|z|^d\quad \textrm{and}\quad |H_j(z)|\geq A_2|z|^{d_j},
	$$
where $d_j=\deg(H_j)$ and $d=\max\{d_j\}$. Suppose that
$z=re^{i\theta}\in\Lambda$, and that $r$ is large enough.
Let $A=\max\{d,2\}$. We choose $c>\frac{4A}{a}$ in Lemma~\ref{indicator-lemma},
and find that there exists a unique index $j_0$ such that
    $$
    \left|f(z)e^{-w_{j_0}z}\right| \geq |H_{j_0}(z)|-A_1mr^de^{-4A\log r},
	$$
from which
	$$
	\frac{|f(z)|}{\left|H_{j_0}(z)e^{w_{j_0}z^q}\right|}\geq 1-\exp\left(-A\log r\right),\quad z\in\Lambda\setminus\Ce.
	$$
On the other hand,
	$$
	\frac{|f(z)|}{\left|H_{j_0}(z)e^{w_{j_0}z^q}\right|}
	\leq 1+\sum_{k\neq j_0}\frac{|H_k(z)|}{|H_{j_0}(z)|}\left|e^{(w_k-w_{j_0})z^q}\right|
	\leq 1+\exp\left(-A\log r\right),
	$$
where $z\in\Lambda\setminus\Ce(f,k)$. This completes the proof of \eqref{f}.

(2) In order to prove \eqref{f-prime}, we first notice that
	$$
	f'=G_0(z)e^{w_0z^q}+G_1(z)e^{w_1z^q}+\cdots  +G_m(z)e^{w_mz^q},
	$$
where $G_j(z)=H'_j(z)+qw_jz^{q-1}H_j(z)$ for $j\geq 0$. If $G_j(z)\equiv 0$, then
either $H_j(z)$ is a constant and $w_j=0$ or $\rho(H_j)=q$. The latter is clearly
impossible, while the former is possible only in the case when $j=0$.
Hence $f'$ is an exponential polynomial
with coefficients being of order $\leq q-p$, and shares the leading coefficients with $f$, except possibly $w_0$.
Thus, by considering separately the cases $j_0\neq 0$ and $j_0=0$, the proof of Part (1) applies to $f'$,
and we obtain \eqref{f-prime} for the same index $j_0$ that appears in \eqref{f}.
This completes the proof.
\end{proof}

If any of the coefficients $H_j$ in \eqref{standard-f} is transcendental, it also can be represented in an analogous normalized form as $f$, where the coefficients are either polynomials or exponential polynomials. Proceeding from one generation of transcendental coefficients to the next, the order
of growth decreases at least by one. Obviously there are at most $q-1$ generations of transcendental descendant coefficients all together.

Let $\{z_n\}$ denote the sequence of zeros of $f$, of the transcendental coefficients $H_j$ of $f$, and of all transcendental descendants of these coefficients, listed according to the multiplicities and ordered according to increasing moduli. Let then $\Ce_0(f,k)$ denote the collection of all discs $|z-z_n|\leq 1/k(|z_n|)$, and suppose that a  suitably large disc $|z|<R$ is also included in $\Ce_0(f,k)$. Clearly,
$\Ce(f,k)$ is a sub-collection of $\Ce_0(f,k)$.

Let $\Theta$ denote the set consisting of the critical angles of $f$
 in \eqref{standard-f}, the critical angles of the transcendental coefficients $H_j$ of $f$, and the critical angles of all transcendental descendants of these coefficients. Some of the critical angles of $f$ may coincide with those of its coefficients or descendant coefficients. Nevertheless, $\Theta$ is a finite subset of $[0,2\pi)$, so the elements
$\theta_j$ of $\Theta$ can be ordered, say $0\leq \theta_1<\theta_2<\cdots<\theta_l<2\pi$. Set $\theta_{l+1}=\theta_1$.

\begin{lemma}\label{logderivative-polynomial}
Let $f$ be an exponential polynomial in the normalized form \eqref{standard-f}, where
we suppose that $\rho(H_j)\leq q-p$ for $1\leq p\leq q$. Let $\theta_j,\theta_{j+1}$ be two consecutive elements of $\Theta$,
and let $\Lambda$ be the domain in between the domains
$\Lambda_p(\theta_j,c)$ and $\Lambda_p(\theta_{j+1},c)$, where $c>0$
is sufficiently large.
If $z\in\Lambda\setminus\Ce_0(f,k)$,
then there exists a unique index $j_0$ and constants
$C_0,\ldots,C_{q-p-1}\in\C$ such that
    \begin{equation}\label{representation-logderivative}
    \frac{f'(z)}{f(z)}=qw_{j_0}z^{q-1}+C_{q-p-1}z^{q-p-1}+\cdots+C_0+O\big(r^{-1}\big).
    \end{equation}
\end{lemma}

\begin{proof}
We will make use of the proof of Lemma~\ref{one-leading-term} not only for $f$
but also for the coefficients and the descendant coefficients of $f$.
For each function the proof is used, new constants $A>q$ and $c>0$ are found.
Since there are at most finitely many functions
involved, we may choose $A$ and $c$ to be the maxima of all these corresponding coefficients.

Independently on whether $w_{j_0}=0$ or $w_{j_0}\neq 0$,
it follows directly from \eqref{f} and \eqref{f-prime} that
    \begin{equation}\label{log-derivative}
    \frac{f'(z)}{f(z)} = \left(qw_{j_0}z^{q-1}+\frac{H_{j_0}'(z)}{H_{j_0}(z)}\right)(1+\veps_{q-p}(r)).
    \end{equation}
Suppose that $H_{j_0}(z)$ is a polynomial. Then \eqref{log-derivative} yields
    \begin{eqnarray*}
    \frac{f'(z)}{f(z)}
    &=& qw_{j_0}z^{q-1}+O\Big(r^{-1}(1+\veps_{q-p}(r))\Big)\\
    &&+O\Big(\exp\Big((q-1)\log r-Ar^{q-p}\log r\Big)\Big)\\
    &=& qw_{j_0}z^{q-1}+O\Big(r^{-1}\Big),
    \quad z\in\Lambda\setminus\Ce_0(f,k),
    \end{eqnarray*}
because of $A>q$. This proves \eqref{representation-logderivative} in the case when $H_{j_0}(z)$ is a polynomial.

Suppose next that $H_{j_0}(z)$ is transcendental. Then $q-p\geq 1$, so we may
write $H_{j_0}(z)$ in the normalized form with coefficients being either exponential polynomials
of order $\leq q-p-1$ or ordinary polynomials.
Analogously as in \eqref{log-derivative}, the proof of Lemma~\ref{one-leading-term},
applied to $H_{j_0}(z)$ instead of $f$, yields
    \begin{equation}\label{K}
    \frac{H_{j_0}'(z)}{H_{j_0}(z)} = \left(C_{q-p-1}z^{q-p-1}+\frac{K'(z)}{K(z)}\right)(1+\veps_{q-p-1}(r)),
    \end{equation}
where $C_{q-p-1}\in\C$ and where $K(z)$ is either an exponential polynomial of order $\leq q-p-1$ or an
ordinary polynomial in $z$.

If $K(z)$ in \eqref{K} is a polynomial, then we combine \eqref{log-derivative} and
\eqref{K} for
    $$
    \frac{f'(z)}{f(z)} = qw_{j_0}z^{q-1}+C_{q-p-1}z^{q-p-1}+O\Big(r^{-1}\Big),\quad z\in\Lambda\setminus\Ce_0(f,k).
    $$
If $K(z)$ is transcendental, then we continue inductively in this fashion by reducing the order
on each step by one.
Eventually the descendant coefficient must reduce down to a polynomial, and, as such, its
logarithmic derivative is of growth $O\left(r^{-1}\right)$. This gives us the representation
\eqref{representation-logderivative}.
\end{proof}

Finally, we remind the reader of the following well-known standard growth estimate for logarithmic derivatives by Gundersen~\cite{Gundersen}.

\begin{lemma}\label{ld-lemma}
Let $f$ be a meromorphic function, let $k,j$ be integers such that $k>j\geq 0$, and let $\alpha>1$. Then there exists a set $E\subset(1,\infty)$ that has finite logarithmic measure, and there exists a constant
$A > 0$ depending only on $\alpha,k,j$, such that for all $z$ satisfying $|z|\not\in E\cup [0,1]$, we have
	\begin{equation}\label{logderivative-EP}
	\left|\frac{f^{(k)}(z)}{f^{(j)}(z)}\right|
    \leq A\left(\frac{T(\alpha r,f)}{r}
    +\frac{n_j(\alpha r)}{r}\log^{\alpha}r \log n_j(\alpha r)\right)^{k-j},
	\end{equation}
where $r = |z|$ and $n_j(r)$ denotes the number of zeros and poles of $f^{(j)}$ in $|z|<r$.	
\end{lemma}


\section{Proof of Theorem \ref{length}}\label{proof1.sec}


The first assertion $n_{\lambda_p}(r)=O\left(r^{q-p}+\log r\right)$
is a simple consequence of \eqref{Stein-thm-p} and the standard
estimate $n(r)\leq (\log 2)^{-1}\left(N(2r)-N(r)\right)$ between
any counting function $n(r)$ and its integrated counterpart $N(r)$.
Thus it suffices to prove \eqref{zeros-n}, which is non-trivial.

Let $\Gamma$ denote any piecewise smooth positively oriented Jordan curve, and let $n(\Gamma)$ denote the zeros of $f$ in the domain bounded by $\Gamma$. By the argument principle, we have
	\begin{equation}\label{ArgumentP}
	n(\Gamma)=\frac{1}{2\pi i}\int_\Gamma \frac{f'(z)}{f(z)}\, dz,
	\end{equation}
provided that $f$ has no zeros on $\Gamma$. In addition, the curve $\Gamma$ needs to be separated from the zeros of $f$ so that we can use the minimum modulus estimate in Corollary~\ref{minmod-cor} as well as its further consequences in Section~\ref{lemmas-Polya}. Hence we need a smart choice for $\Gamma$.

In order to simplify the situation, we may suppose that the positive real axis is a
critical ray for $f$ by appealing to the
exponential polynomial $g(z)=f(e^{-i\theta^*}z)$, if necessary. Thus, we assume that $\theta^*=0$
is a critical angle for $f$, determined by its leading coefficients $w_j,w_k$ as in \eqref{cusp}.
It follows that $w_j$ and $w_k$ are the unique dominant leading coefficients of $f$
on the sides of the positive real axis. Without loss of generality, we may assume that $w_j$ determines the dominant term of $f$ below the $x$-axis
and that $w_k$ determines the dominant term of $f$ above the $x$-axis.
The uniqueness of the constants $w_j,w_k$ carries over
to the application of Lemma~\ref{logderivative-polynomial}.

Let $k(x)=x^{q+p-1}\log x$ and $r_n=1/k(|z_n|)$ for $x\geq e$ and $|z_n|\geq e$. For this particular $k$, we let $\Ce_0(f,k)$ denote the set discussed in Section~\ref{lemmas-Polya}. Then \eqref{convergence} clearly holds, and so, by Lemma~\ref{meet-lemma}, we can find $c>0$ such that $\partial\Lambda_p(0,c)$ meets at most finitely many discs $|z-z_n|\leq r_n$. This constant $c$ also takes into account the $p$-generalization \eqref{Stein-thm-p} of the result by Steinmetz. Denote by $\Lambda^+(c)$ and $\Lambda^-(c)$ the portions of $\Lambda_p(0,c)$ in the upper half-plane and the lower half-plane, respectively.

     \begin{figure}[h]\label{Integralpath}
    \begin{center}
    \begin{tikzpicture}
    \draw[->](0,0)--(10,0)node[left,below]{$x$};
    \draw[->](1,-4)--(1,4)node[right]{$y$};
   \draw[->][domain=10:8] plot(\x,{10*ln(\x)/\x^1.3})node[right,above=5pt,font=\tiny]{$\Gamma_3:z=te^{\frac{-ic\log t}{t^p}}$};
    \draw[-][domain=8:1] plot(\x,{10*ln(\x)/\x^1.3});
     \draw[->][domain=1:6] plot(\x,{-10*ln(\x)/\x^1.3});
     \draw[-][domain=6:10] plot(\x,{-10*ln(\x)/\x^1.3})node[right,below,font=\tiny]{$\Gamma_1:z=te^{\frac{-ic\log t}{t^p}}$};
    \draw[-,dashed] (8.6,-1.335) arc(-24:24:3.2)node[right=100pt,below=20pt,font=\tiny]{$\Gamma_2:z=re^{\theta}$};
     \draw[-,dashed] (5,-2) arc(-35:35:3.5)node[left=4pt,below=40pt,font=\tiny]{$\Gamma_4:z=Re^{\theta}$};
    \end{tikzpicture}
    \end{center}
    \caption{The domain bounded by the curve $\Gamma=\Gamma_1+\Gamma_2+\Gamma_3+\Gamma_4$.}
    \end{figure}
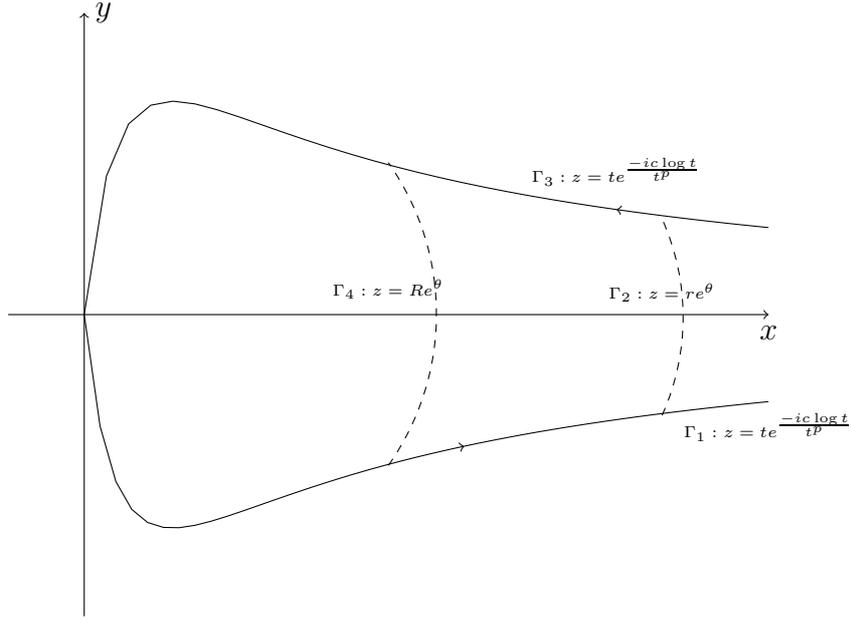	

Let $\alpha>1$, and let $E\subset (1,\infty)$ denote the exceptional set of finite logarithmic measure in Lemma~\ref{ld-lemma}, where $f$ is our exponential polynomial of order $q$, and $k=1$, $j=0$. Choose $R\in (e,\infty)\setminus E$ large enough such that the curves $\Lambda^+(c)$ and $\Lambda^-(c)$ do not meet any of the discs $|z-z_n|\leq r_n$ for $r\geq R$. Then choose any $r\in (R,\infty)\setminus E$.
Finally, let $\Gamma=\Gamma_1+\Gamma_2+\Gamma_3+\Gamma_4$, where
	$$
	\begin{array}{cll}
	\Gamma_1 &: z=te^{-ic\frac{\log t}{t^{p}}}, &\textnormal{where $t$ goes from $R$ to $r$},\\
	\Gamma_2 &: z=re^{i\theta},
	&\textnormal{where $\theta$ goes from $-c\frac{\log r}{r^{p}}$ to $c\frac{\log r}{r^{p}}$},\\
	\Gamma_3 &: z=te^{ic\frac{\log t}{t^{p}}}, &\textnormal{where $t$ goes from $r$ to $R$},\\
	\Gamma_4 &: z=Re^{i\theta},
	&\textnormal{where $\theta$ goes from $c\frac{\log R}{R^{p}}$ to $-c\frac{\log R}{R^{p}}$}.
	\end{array}
	$$
The idea is to keep $R$ fixed and eventually to let $r$ increase. Due to the construction, it is clear that $\Gamma$ has no zeros of $f$, even when $r\in (R,\infty)\setminus E$ is arbitrarily large. We find that
    $$
	I:=\int_{\Gamma_1} w_jqz^{q-1}\,dz=\int_{R}^{r}w_jqt^{q-1}e^{-iqc\frac{\log t}{t^p}}
	\left(1-\frac{1-p\log t}{t^{p+1}}ci\right)\,dt.
	$$
If we set $w_j=a_j+ib_j$, then
	\begin{equation*}
    \begin{split}
	\Im I=&\int_{R}^{r} b_jq t^{q-1} \bigg\{\cos\left(qc\frac{\log t}{t^p}\right)
	-\sin\left(qc\frac{\log t}{t^p}\right)\frac{1-p\log t}{t^{p+1}}c\bigg\}\,dt\\
	&-\int_{R}^{r} a_jq t^{q-1}\bigg\{\sin\left(qc\frac{\log t}{t^p}\right)+
	\cos\left(qc\frac{\log t}{t^p}\right)\frac{1-p\log t}{t^{p+1}}c\bigg\}\,dt.
	\end{split}
	\end{equation*}
Using the standard estimates $1-x^2/2\leq \cos x\leq 1$ and
$x-x^3/6\leq \sin x\leq x$, we get
	$$
	\Im I=b_jr^q+O\left(r^{q-p}\log r\right).
	$$
Using Lemma~\ref{logderivative-polynomial}, it follows that
    $$
    \Im\int_{\Gamma_1}\frac{f'(z)}{f(z)}\, dz=b_jr^q+O\left(r^{q-p}\log r\right).
    $$
Analogously, if $w_k=a_k+ib_k$, then
	$$
    \Im\int_{\Gamma_3}\frac{f'(z)}{f(z)}\, dz=-b_kr^q+O\left(r^{q-p}\log r\right).
    $$
The minus-sign here is a result of integrating in the opposite direction.
Next we apply Lemma~\ref{ld-lemma} together with Corollary~\ref{n-cor}, and obtain the estimate
	$$
	\left|\int_{\Gamma_2}\frac{f'(z)}{f(z)}\, dz +\int_{\Gamma_4}\frac{f'(z)}{f(z)}\, dz\right|
	=O\left(r^{q-p}\log^{2+\alpha}r\right).
	$$
By putting everything together, we conclude that
	\begin{equation}\label{Gl}
	\frac{1}{2\pi i}\int_{\Gamma}\frac{f'(z)}{f(z)}\, dz	=
	\frac{b_j-b_k}{2\pi}r^q+O\left(r^{q-p}\log^{2+\alpha} r\right).
	\end{equation}
	
Keeping \eqref{ArgumentP} in mind, the left-hand side of \eqref{Gl}
is the number of zeros $n(r,\Gamma)$ of $f$ in a domain bounded
by the closed curve $\Gamma$ that depends on $r$. Since counting functions are always
non-negative, it follows that $b_j-b_k=|b_j-b_k|$. The situation would
be symmetric if we would have chosen the leading coefficients $w_j$
and $w_k$ the other way around in the reasoning above. Moreover,
from \eqref{cusp}, we get $\Re{w_je^{iq\theta^*}}=\Re{w_ke^{iq\theta^*}}$,
which can be written as
    $$
    a_j\cos q\theta^*-b_j\sin q\theta^*=a_k\cos q\theta^*-b_k\sin q\theta^*.
    $$
Since we have chosen $\theta^*=0$, it follows that $a_j=a_k$.
Therefore, we have $|w_j-w_k|=|b_j-b_k|$, so that \eqref{Gl}
now reads as
	$$
	n(r,\Gamma)=
	\frac{|w_j-w_k|}{2\pi}r^q
	+O\left(r^{q-p}\log^{2+\alpha} r\right).
	$$
The assertion follows from this via Lemma~\ref{T-lemma} by writing
$\alpha =1+\veps$.

\bigskip

\noindent
\emph{J.~M.~Heittokangas}\\
\textsc{University of Eastern Finland, Department of Physics and Mathematics,
P.O.~Box 111, 80101 Joensuu, Finland}\\
\textsc{Taiyuan University of Technology,
Department of Mathematics,
Yingze West Street No.~79, Taiyuan 030024, China}\\
\texttt{e-mail:janne.heittokangas@uef.fi}
\medskip

\noindent
\emph{Z.-T.~Wen}\\
\textsc{Shantou University, Department of Mathematics, Daxue Road No.~243, Shantou 515063, China}\\
\textsc{Taiyuan University of Technology,
Department of Mathematics,
Yingze West Street No.~79, Taiyuan 030024, China}\\
\texttt{e-mail:zhtwen@stu.edu.cn}


\begin{thebibliography}{00}

\bibitem{CGHR}
	Chuaqui M., J.~Gr\"ohn, J.~Heittokangas and J.~R\"atty\"a,
	\emph{Zero separation results for solutions of second order linear 		differential equations}.
	Adv.~Math.~\textbf{245} (2013), 382--422.

\bibitem{Dick0}
	Dickson D.~G.,
	\emph{Expansions in series of solutions of linear
	difference-differential
	and infinite order equations with constant coefficients.}
	Mem.~Amer.~Math.~Soc., No.~23 (1957), 72 pp.
	
\bibitem{Dick}
    Dickson D.~G.,
    \emph{Asymptotic distribution of zeros of exponential sums}.
    Publ.~Math.~Debrecen \textbf{11} (1964), 297--300.

\bibitem{DHW}
	Ding J., J.~Heittokangas and Z.-T.~Wen,
	\emph{From $\E$-sets to $R$-sets and beyond}.	
	Manuscript in preparation.

\bibitem{GM}
    Gackstatter F.~and G.~P.~Meyer,
    \emph{Zur Wertverteilung der Quotienten von Exponentialpolynomen}.
     Arch.~Math.~(Basel) \textbf{36} (1981), no.~3, 255--274.

\bibitem{Gundersen}
	Gundersen G.~G.,
	\emph{Estimates for the logarithmic derivative of a meromorphic
	function, plus similar estimates.}
	J.~London Math.~Soc.~(2) \textbf{37} (1988), no.~1, 88--104.

\bibitem{Gundersen2}
	Gundersen G.~G.,
	\emph{Finite order solutions of second order linear
	differential equations}.
	Trans.~Amer.~Math.~Soc.~\textbf{305} (1988), no.~1, 415--429.

\bibitem{Hayman}
	Hayman W.,
	\emph{Meromorphic Functions}.
	Oxford Mathematical Monographs, Clarendon Press, Oxford, 1964.

\bibitem{HITW}
	Heittokangas J., K.~Ishizaki, K.~Tohge and Z.-T.~Wen,
	\emph{Zero distribution and division results for
	exponential polynomials}.
	Israel J.~Math.~\textbf{227} (2018), 397--421.


\bibitem{GOP}
	Heittokangas J., I.~Laine, K.~Tohge and Z.-T.~Wen,
	\emph{Completely regular growth solutions of second order
	complex linear differential equations}.
	Ann.~Acad.~Sci.~Fenn.~\textbf{40} (2015), no.~2, 985--1003.

\bibitem{HN}
	Helmrath W.~and J.~Nikolaus,
	\emph{Ein elementarer Beweis bei der Anwendung der
	Zentralindexmethode auf Differentialgleichungen}.
	Complex Variables Theory Appl.~\textbf{3} (1984), no.~4, 387--396.

\bibitem{Laine}
	Laine I.,
	\emph{Nevanlinna Theory and Complex Differential Equations}.
	De Gruyter Studies in Mathematics, 15. Walter de Gruyter \& Co., Berlin, 1993.

\bibitem{Langer}
	Langer R.~E.,
	\emph{On the zeros of exponential sums and integrals}.
	Bull.~Amer.~Math.~Soc.~\textbf{37} (1931), no. 4, 213--239.

\bibitem{Lax}
	Lax P.~D.,
	\emph{The quotient of exponential polynomials}.
	Duke Math.~J.~\textbf{15} (1948), 967--970.

\bibitem{Levin1}
    Levin B.~Ja.,
    \emph{Distribution of Zeros of Entire Functions}.
    Translated from the Russian by R.~P.~Boas, J.~M.~Danskin, 	
    F.~M.~Goodspeed, J.~Korevaar, A.~L.~Shields and H.~P.~Thielman.
    Revised edition. Translations of Mathematical Monographs, 5.
    American Mathematical Society, Providence, R.I., 1980.

\bibitem{MacColl}
    MacColl L. A.,
    \emph{On the distributions of the zeros of sums of exponentials of polynomials}.
    Trans.~Amer.~Math.~Soc.~\textbf{36} (1934), no.~2, 341--360.

\bibitem{Poly}
    P\'olya G.,
    \emph{Geometrisches \"uber die Verteilung der Nullstellen
    spezieller ganzer Funktionen}.
    Sitz.-Ber.~Bayer.~Akad.~Wiss.~(1920), 285--290.

\bibitem{Poly2}
    P\'olya G.,
    \emph{Untersuchungen \"uber L\"ucken und Singularit\"aten
    von Potenzreihen}.
    Math.~Z.~\textbf{29} (1929), no.~1, 549--640.

\bibitem{Ritt}
	Ritt J.~F.,
	\emph{On the zeros of exponential polynomials}.
	Trans.~Amer.~Math.~Soc.~\textbf{31} (1929), no.~4, 680--686.

\bibitem{Sch}
    Schwengeler E.,
    \emph{Geometrisches \"uber die Verteilung der Nullstellen spezieller
    ganzer Funktionen (Exponentialsummen)}.
    Diss.~Z\"urich, 1925.

\bibitem{Shapiro}
    Shapiro H.~S.,
    \emph{The expansion of mean-periodic functions in
    series of exponentials}.
    Comm.~Pure Appl.~Math.~\textbf{11} (1958), 1--21.

\bibitem{Stein1}
	Steinmetz N.,
	\emph{Zur Wertverteilung von Exponentialpolynomen}.
	Manuscripta Math.~\textbf{26} (1978/79), no.~1--2, 155--167.

\bibitem{Tsuji}
	Tsuji M.,
	\emph{Potential Theory in Modern Function Theory}.
	Reprinting of the 1959 original.
	Chelsea Publishing Co., New York, 1975.

\bibitem{VPT}
	Voorhoeve M., A.~J.~van der Poorten and R.~Tijdeman,
	\emph{On the number of zeros of certain functions}.
	Ned.~Akad.~Wer.~Proc.~Ser.~A \textbf{78} (1975), 407--416.

\bibitem{W1}
    Wittich H.,
    \emph{Bemerkung zur Wertverteilun von Exponentialsummen}.
    Arch.~Math. (Basel) \textbf{4} (1953), 202--209.

\end{thebibliography}
\end{document}